\pdfoutput=1 
\documentclass[12pt,reqno,letterpaper]{amsart}
\usepackage{KST-macros}

\title{Homotopy versus isotopy: 2-spheres in 5-manifolds}

\author[D. Kosanovi\'{c}]{Danica Kosanovi\'{c}}
\address{Dept.\ of Mathematics, ETH Z\"urich, Switzerland}
\email{danica.kosanovic@math.ethz.ch}

\author[R. Schneiderman]{Rob Schneiderman}
\address{Dept.\ of Mathematics, Lehman College, City University of New York, Bronx, NY}
\email{robert.schneiderman@lehman.cuny.edu}

\author[P. Teichner]{Peter Teichner}
\address{Max-Planck-Institut f\"ur Mathematik, Bonn, Germany}
\email{teichner@mac.com}

\begin{document}

\maketitle

\begin{abstract}
    In this note we give a complete obstruction for two  homotopic embeddings of a 2-sphere into a 5-manifold to be isotopic. The results are new even though the methods are classical, the main tool being the elimination of double points via a level preserving Whitney move in codimension~$3$. Moreover, we discuss how this recovers a particular case of a result of Dax on metastable homotopy groups of embedding spaces. It follows that ``homotopy implies isotopy'' for 2-spheres in simply-connected 5-manifolds and for 2-spheres admitting algebraic dual 3-spheres.
\end{abstract}



\section{Introduction and results}

A curious consequence of our generalizations \cite{ST-LBT,KT-4dLBT} of the 4-dimensional Light Bulb Theorems of David Gabai~\cite{Gabai-spheres,Gabai-disks} is that homotopic 2-spheres $R,R'\colon S^2 \hra M$, embedded in a 4-manifold $M$ with a common dual sphere, are smoothly isotopic in $M$ if and only if they are isotopic in the 5-manifold $M\times \R$, see \cite[Cor.1.5]{ST-LBT}.
The complete isotopy obstruction in \cite[Thm.1.1]{ST-LBT} is given by the Freedman--Quinn invariant 
\[
    \fq(R,R')\coloneqq [\mu_3(H)]\in \faktor{\F_2T_M}{\mu_3(\pi_3M)},
\]
where $\mu_3(H)$ is the self-intersection invariant of the track $H\colon S^2 \times [0,1] \imra M \times \R \times [0,1]$ of a generic homotopy between $R$ and $R'$ in $M \times \R$. Moreover, $\F_2T_M$ is the $\F_2$-vector space with basis $T_M\coloneqq\{g\in\pi_1M \mid g^2=1\neq g\}$, the set of involutions in $\pi_1M$. It turns out that Wall's self-intersection invariant $\mu_3$ 
also gives a homomorphism $\mu_3\colon\pi_3M\to \F_2T_M$,
whose cokernel eliminates the choice of homotopy in the definition of $\fq$. Michael Freedman and Frank Quinn introduced this invariant in the book \cite[Ch.10]{FQ}, while studying topological concordance classes of embedded 2-spheres in 4-manifolds. 

This isotopy classification also follows from \cite[Thm.1.1]{KT-4dLBT}, via a more powerful invariant, due originally to Jean-Pierre Dax \cite{Dax}, which detects relative isotopy classes of neatly embedded 2-disks having a common dual in $\partial M$. Dax extends the parametrized double-point elimination method of Andr\'e Haefliger~\cite{Haefliger-bulletin,Haefliger-plong}, which is in turn an extension of the Whitney trick \cite{Whitney}. Haefliger's results were used by Lawrence Larmore~\cite[Thm.6.0.1]{Larmore} to show a special case of Dax's result, see equation \eqref{eq-intro:Dax} below.

In the current paper we consider the ``homotopy versus isotopy'' question for $2$--spheres in general 5--manifolds and show that there is again a self-intersection invariant of a homotopy in a quotient of the group ring of the ambient fundamental group which detects isotopy classes. The dimensions under consideration here are right at the transition entering high-dimensional topology, with successful Whitney moves generally available in the presence of vanishing algebraic obstructions. With this in mind, our exposition will be aimed at describing this transition from the point of view of the low-dimensional topologist, rather than starting by presenting results in full generality. In particular, we will:
\begin{enumerate}
    \item explain how the isotopy classification can be described by self-intersection invariants of homotopies, using a level-preserving Whitney trick;
    \item explain how Dax's work recovers the same result, from the perspective of space level techniques and homotopy groups of embedding spaces.
\end{enumerate}
Both approaches can be generalized to describe isotopy classifications for compact $n$-manifolds embedded in $(2n+1)$-manifolds. In upcoming work~\cite{KST-Dax} we recast \cite{Dax} in full generality in this language; see Theorem~\ref{thm:main-Dax-part} below. 

We next give a quick outline of the main results, and refer to the rest of the introduction for details.
For an embedded sphere $U\colon S^2\hra N^5$ in a $5$--manifold $N$, thought of as the ``Unknot'', we will define the set $\A_{U_*}$ as a certain quotient of the group ring $\Z\pi_1N$, see Definition~\ref{def:unbased-A_F}. 
The image in $\A_{U_*}$ of the self-intersections $\mu_3(H)$ of a generic track $H\colon S^2 \times I \imra N \times I$ of a homotopy between $U$ and $R\colon S^2\hra N$ will by design only depend on $U$ and $R$, not on $H$.
Denoting this invariant by $\fq_{U_*}(R)\coloneqq [\mu_3(H)]\in \A_{U_*}$, a basic statement of our main result is the following corollary of Theorem~\ref{thm:free}:
\begin{cor}\label{cor:free}
    Homotopic spheres $U$ and $R$  are isotopic if and only ${\fq_{U_*}(R)=0\in\A_{U_*}}$. Moreover, any element in $\A_{U_*}$ is realized as $\fq_{U_*}(R)$ for an embedded sphere $R$.
\end{cor}
We note that in this 5-dimensional setting, the result does not require any dual spheres (unlike in four dimensions), cf.\ Corollary~\ref{cor:5d-LBT}. 
If $\pi_1N$ is trivial, then $\A_{U_*} =\{0\}$ and we get:
\begin{cor}\label{cor:simply-connected}
    Homotopy implies isotopy for $2$--spheres in simply-connected $5$--manifolds.
\end{cor}

There is a straightforward proof of Corollary~\ref{cor:simply-connected}, using cusp-cancellation (in dimension 6) and the theorem of John Hudson \cite{Hudson} that in codimension~$>2$ concordance implies isotopy. We will give a self-contained proof of the general classification result, by providing a level-preserving version of the Whitney move in codimension~$>2$ (Proposition~\ref{prop:6d-Whitney}).

If $\pi_1N$ is not trivial then our classification result is similar to the $4$-dimensional setting with common duals, namely $\fq_{U_*}$ gives the unique obstruction for embeddings that are homotopic to be isotopic. Our main work will be in spelling out the precise range of this obstruction and showing that all these values are realized. A new issue that arises in the current setting is the distinction between based and free homotopies, whereas the assumption of common duals in the 4-dimensional setting essentially allowed for consideration of only based homotopies (cf.~\cite[Thm.6.1]{Gabai-spheres}, \cite[Lem.2.1]{ST-LBT}). We will occasionally emphasize this issue by applying the adjective ``free'' to the terms ``homotopy'' and ``isotopy'', even though by traditional definitions it would suffice to just omit the adjective ``based''.

\subsection{2-knots in 5-manifolds}\label{subsec:intro-square-diagram}
We now turn to precise formulations of our main results, working in the smooth oriented category throughout.
Fixing a basepoint in such a 5-manifold $N$, and a basepoint in $S^2$, we have
the following commutative diagram which will guide the discussion of our invariants:
\[
\begin{tikzcd}
    \pi_0 \Emb_*(S^2, N) \rar[two heads]{p_*} \dar[two heads]{\text{mod }\pi_1N} & \dar[two heads]{\text{mod } \pi_1N}  \pi_0\Map_*(S^2, N) \\
    \pi_0\Emb(S^2, N) \rar[two heads]{p} & \pi_0\Map(S^2, N) 
\end{tikzcd}
\]
Here $\pi_0 \Emb_*(S^2, N)\coloneqq \{\text{based embeddings } S^2 \hra N \} / \text{based isotopy}$ is the embedded version of $\pi_0\Map_*(S^2, N)=\pi_2N$,
and 
 $\pi_0 \Emb(S^2, N)\coloneqq \{\text{embeddings } S^2 \hra N \} / \text{free isotopy}$ is the embedded version of $\pi_0\Map(S^2, N)=[S^2,N]$.
Both horizontal arrows forget the fact that we have embeddings, and 
the vertical arrows divide out the $\pi_1N$-actions (using embedded tubes along closed paths at the basepoint on the left hand side).
Both $p$ and $p_*$ are surjective (that is, any (based) map is homotopic to a (based) embedding) by general position: the dimension of the double point set is $5-6=-1$ (so generically this set is empty), since codimensions of generic intersections add.

Along with using the label $*$ for based objects, our notational convention is to use a bracket to denote the homotopy class of an embedded object, which is otherwise considered up to isotopy. For example, the upper horizontal map $p_*$ sends $R_*\in \pi_0 \Emb_*(S^2, N)$ to $[R_*]\in\pi_0 \Map_*(S^2, N)$,
and the left vertical map sends $R_*$ to $R\in \pi_0\Emb(S^2, N)$. 
We are ultimately interested in the fibers of $p$, but it turns out to be convenient to first understand the fibers of $p_*$.

\subsection{The based isotopy invariant}\label{subsec:intro-based-isotopy}

As recalled in \eqref{eq-def:mu3-H} below, the quotient
\[
    \A\coloneqq \faktor{\Z\pi_1N}{\langle g+g^{-1},1\rangle}
\]
of the integral fundamental group ring $\Z\pi_1N\cong\Z\pi_1(N\times I)$
is the usual target for the self-intersection invariant
\begin{equation}\label{eq:mu3-imm}
    \mu_3\colon\big\{
      \text{simply-connected $3$-manifolds immersed in the $6$-manifold $N\times I$}
    \big\}\ra \A.
\end{equation}
Let us fix a based embedding $U_*\colon S^2 \hra N$, and define a homomorphism of abelian groups $\phi_{[U_*]}\colon \pi_3N \longrightarrow \A$ by
\begin{equation}\label{eq-def:phi}
    \phi_{[U_*]}(A)\coloneqq \mu_3(A) + [\lambda_N(A,[U_*])]
\end{equation}
where $\mu_3$ denotes the self-intersection invariant on $\pi_3(N\times I)\cong\pi_3N$, and $\lambda_N$ is the intersection pairing between $\pi_3N$ and $\pi_2N$ taking values in $\Z\pi_1N$. 


\begin{defn}\label{def:based-A_F} 
$
\A_{[U_*]} \coloneqq \faktor{\A}{\phi_{[U_*]}(\pi_3N)}
$
\end{defn}

Note that the abelian group $\A_{[U_*]}$ only depends on the based homotopy class  $[U_*]\in\pi_2N$. 
We now consider the fiber of $p_*$ over $[U_*]\in\pi_2N$:
\[
    p_*^{-1}([U_*]) \coloneqq \faktor{\big\{R_*\colon S^2\hra N \mid R_* \text{ is based homotopic to }  U_*\big\}}{\text{based isotopy}}\;.
\]
\begin{defn}\label{def:based-fq}
    For $R_*\in p_*^{-1}[U_*]$, let $H_*\colon S^2\times I\imra N\times I$ be a generic track of a based homotopy from $U_*$ to $ R_*$, and define:
    \[
        \fq_{[U_*]}(R_*)\coloneqq [\mu_3(H_*)]\in\A_{[U_*]}.
    \]
\end{defn}
By the following theorem $\fq_{[U_*]}(R_*)$ does not depend on the choice of homotopy, and vanishes if and only if $R_*$ and $U_*$ are based isotopic.

\begin{thm}\label{thm:based-fq}
    The map $\fq_{[U_*]}\colon p_*^{-1}[U_*]\ra \A_{[U_*]}$ is a bijection, whose inverse is given by a geometric action on $U_*$.
\end{thm}
The action of $g\in \pi_1N$ on  $U_*\in \pi_0 \Emb_*(S^2, N)$ is by a ``finger move'' along $g$, which in this setting is an ambient connected sum of $U_*$ with its meridian sphere $m_{U_*}$ along a tube following a loop representing $g$. Elements in the group ring act by multiple finger moves, which turn out to involve signs and preserve the relations in the quotient $\A$ of the group ring (see section~\ref{subsec:geo-action}). The proof of Theorem~\ref{thm:based-fq} shows the following.


\begin{cor}\label{cor:based-action} 
    The abelian group $\A$ acts on $\pi_0 \Emb_*(S^2, N)$ compatibly with the $\pi_1N$-actions, preserving $p_*$ and transitively on its fibers, with the stabilizer of $U_*$ equal to $\phi_{[U_*]}(\pi_3N)$. 
\end{cor}

\subsection{The free isotopy invariant}\label{intro:free-isotopy}
Now consider \emph{free} homotopy versus isotopy, i.e.\ the set:
\[
    p^{-1}[U] \coloneqq \faktor{\big\{R\colon S^2\hra N \mid R \textrm{ is freely homotopic to }  U\big\}}{\text{free isotopy}}
\]
for $U\colon S^2\hra N$ a fixed embedding in $\Emb(S^2,N)\subset\Map(S^2,N)$, and choose $U(e)\in N$ as the basepoint for $N$ (where $e$ denotes the basepoint for $S^2$).

To define the target of an invariant that characterizes $p^{-1}[U]$ we will define an
affine action on $\A_{[U_*]}$ (the range of the bijection in the based setting of Theorem~\ref{thm:based-fq}) by the group
\[
    \Stab[U_*]\coloneqq \{s\in\pi_1N\colon s\cdot[U_*]=[U_*]\},
\]
that is the stabilizer subgroup of $[U_*]\in\pi_2N$ of the usual action of $\pi_1N$ on $\pi_2N$.

Recall that an \emph{affine transformation} $T$ of an abelian group $A$ is given by an endomorphism $\ell$ and a translation $a_0$ of $A$, i.e.\ $T(a) = a_0 + \ell(a)$, where $a_0=T(0)$. An \emph{affine action} of a group on $A$ is a homomorphism to the group of affine transformations of $A$. In our case, the linear action of $s\in \Stab[U_*]$ will be $a\mapsto sas^{-1}$, whereas the translational part will be given by $U_s$, both of which we explain next.

Firstly, we claim that the linear action $(s,a)\mapsto s as^{-1}$ of $\Stab[U_*]$ on  $\Z\pi_1N$ descends to $\A_{[U_*]}$: for $A\in\pi_3N$ we have $\mu_3(g\cdot A) = g \mu_3(A) g^{-1}$ and $\lambda_N(g\cdot A, [U_*]) = g\lambda_N(A, [U_*])$, so if $g\cdot [U_*]=[U_*]$ then the last expression also equals $g\lambda_N(A,  [U_*])g^{-1}$, implying $g\phi_{[U_*]}(A)g^{-1}=\phi_{[U_*]}(g\cdot A)\in \phi_{[U_*]}(\pi_3N)$. 

Secondly, for $s\in \Stab[U_*]$ there is a generic track $J_s\colon S^2 \times I \imra N\times I$ of a free self-homotopy of $U_*$ such that the projection of $J_s(e,-)$ to $N$ represents $s$. It turns out that
\[
    U_s \coloneqq [\mu_3(J_s)] \in \A_{[U_*]}
\]
only depends on $s$ and the isotopy class $U_*$ (and not on $J_s$, see Lemma~\ref{lem:J_s}). It is easy to show that under concatenation, for $s,r\in \Stab[U_*]$ this behaves as:
\[
    U_{s\cdot r} = U_s + sU_rs^{-1}.
\]
This implies that the formula
\begin{equation}\tag{aff}\label{eq:aff}            
    {}^sa \coloneqq U_s + sas^{-1} 
\end{equation}
satisfies
\begin{equation}
    {}^{sr}a\coloneqq U_{sr} + srar^{-1}s^{-1}= 
    U_s+sU_rs^{-1} + srar^{-1}s^{-1}
    = {}^s(U_r+rar^{-1})= {}^s({}^ra).
\end{equation}
In other words, the composition in the group $\Stab[U_*]$ turns into the composition of affine transformations, so $s\mapsto {}^sa$ defines an affine action of $s\in\Stab[U_*]$ on $a\in\A_{[U_*]}$. 

\begin{defn}\label{def:unbased-A_F}
    Denote by $\A_{U_*}$ the quotient set of this affine action by $\Stab[U_*]$ on $\A_{[U_*]}$. 
\end{defn}

Note that the definition of $\A_{U_*}$  depends on the based isotopy class $U_*$, a fixed basing of $U$. However, the following result shows that it 
characterizes the set of \emph{free} isotopy classes of embeddings homotopic to $U$. 
\begin{thm}\label{thm:free}
    There is a bijection
    \begin{align*}
        \fq_{U_*}\colon p^{-1}[U] &\ra \A_{U_*}\\
        R&\mapsto[\mu_3(H)],
    \end{align*}
    where $H$ is a generic track of any free homotopy from $U$ to $R$. 
\end{thm}
Here both the computation of $\mu_3(H)$ and the definition of $\A_{U_*}$ use the same basing $U_*$, and the inverse of the bijection is defined using the same geometric action as in Corollary~\ref{cor:based-action}.

Curiously, in the based setting of Theorem~\ref{thm:based-fq}, both the target $\A_{[U_*]}$ and the set $p_*^{-1}[U_*]$ only depend on the based homotopy class $[U_*]\in\pi_2N$, whereas in the free setting of Theorem~\ref{thm:free}, the target $\A_{U_*}$ depends on the \emph{based isotopy} class $U_*\in\pi_0 \Emb_*(S^2, N)$, 
while the set $p^{-1}[U]$ depends on the \emph{free homotopy} class $[U]\in\pi_0\Map(S^2, N)$ of the embedding $U$.

\subsubsection*{Example}
If $U_*$ is the trivial 2-sphere then $U_s=0$ for all $s \in\Stab[U_*]= \pi_1N$ because any self-homotopy $J_s$ can be chosen to be a self-isotopy. As a consequence:

\begin{cor}[Null-homotopic isotopy classes]\label{cor:null-homotopic}
    Null-homotopic free isotopy classes of 2-spheres in $N$ are in bijection with $\A_{U_*}={\A_{[U_*]}}/{\pi_1N}$, the quotient of the abelian group
    \[
        \A_{[U_*]}=\faktor{\Z[\pi_1N]}{\langle 1, g+g^{-1}, \mu_3(A):g\in\pi_1N,A\in\pi_3N\rangle},
    \]
    by the conjugation action of $\pi_1N$.
\end{cor}
Note that although the free self-isotopies of $U$ have vanishing self-intersection invariants, 
they still contribute to the indeterminacy of the invariant $\fq_{U_*}$ by arbitrarily conjugating double-point group elements in the computation of $\mu_3(H)$. 

\subsubsection*{Example}
If $\lambda_N(G, [U_*])=1$ for some $G\in\pi_3N$ such that $\mu_3(G)=0\in\A$,
then 
$\phi_{[U_*]}(\pi_3N)=\A$; since for any $\sum g_i\in\Z[\pi_1N]$ we have
\[
    \phi_{[U_*]}\left(\sum g_i \cdot G\right)=0+
    \left[\lambda_N(\sum g_i \cdot G,[U_*])\right]=
    \left[\sum g_i \cdot\lambda_N(G,[U_*])\right]=
    \left[\sum g_i\right].
\]
It follows that in this case $\A_{[U_*]}$ contains a single element, and hence so does $\A_{U_*}$.

\begin{cor}[5-dimensional Light Bulb Theorem]\label{cor:5d-LBT}
    ``Homotopy implies isotopy'' for spheres $U_*\colon S^2\hra N^5$ admitting an ``algebraic dual'' $G\in\pi_3N$ as above.
\end{cor}
See section~\ref{sec:examples} for more examples.

%
\subsection{Isotopy classification via mapping spaces}\label{subsec:intro-Dax}
In section~\ref{sec:Dax-approach} we present a slightly different perspective to the problem of isotopy classification.
Namely, the fibers of the map $p\colon \pi_0\Emb(S^2,N)\to \pi_0\Map(S^2,N)$ can also be determined using the homotopy exact sequence associated to the inclusion of mapping spaces 
$\Emb(S^2,N)\subset\Map(S^2,N)$:
\[\begin{tikzcd}[column sep=small,nodes={scale=0.94}]
    \pi_1(\Map(S^2,N), U)\rar{j} &  \pi_1(\Map(S^2,N),\Emb(S^2,N), U)\rar &
    \pi_0\Emb(S^2,N)\rar[two heads]{p} &
    \pi_0\Map(S^2,N). 
\end{tikzcd}
\]
Here we picked an embedding $U\colon S^2\hra N$ as a basepoint, and the leftmost absolute $\pi_1$ is a group that acts on the relative $\pi_1$ (which is just a set) such that the quotient set is isomorphic to the fiber $p^{-1}[U]$ of $p$ over $[U]\in\pi_0\Map(S^2,N)$. 

The relative $\pi_1$ is the first non-vanishing relative homotopy group, and that is exactly what was computed by Jean-Pierre Dax \cite{Dax}. 
He translated this (and also other relative homotopy groups in the ``metastable range'') to certain bordism groups.
Computations of this bordism group (which is 0-dimensional in the first non-vanishing case) were carried out in \cite{KT-highd}, for the cases $(\Map_\partial(V,X),\Emb_\partial(V,X))$ when embeddings have nonempty boundary condition, and $V$ is 1-connected. In \cite{KST-Dax} we extend this to closed and disconnected manifolds,
and specializing \cite{KST-Dax} to $V=S^2$ and $d=5$ leads to Theorem~\ref{thm:main-Dax-part}: there is a bijection
\begin{align}\label{eq-intro:Dax}
    \Da \colon \pi_1(\Map(S^2,N),\Emb(S^2,N), U) 
    &\ra \A\\
    h &\mapsto \mu_3(H),\nonumber
\end{align}
where $H$ is a homotopy from $U$ to an embedding that represents $h$, and $\mu_3(H)\in\A$ is the self-intersection invariant of a generic track $H$ of the homotopy, as in~\eqref{eq:mu3-imm} above.

In order to compute the set $p^{-1}[U]$ from this viewpoint it remains to understand the action of the absolute $\pi_1$ on the relative $\pi_1$ in the sequence displayed above. Firstly, we use the fibration sequence $\Map_*(S^2,N)\to \Map(S^2,N)\to N$, where the first map is the inclusion $i$, and the second map evaluates at the basepoint $e\in S^2$, to obtain the exactness in the column of the following diagram:
    \begin{equation}\label{eq:big-diagram}
        \begin{tikzcd}
       \pi_3N\cong\pi_1\big(\Map_*(S^2,N),U_*\big)\dar{i} & & \\
        \pi_1\big(\Map(S^2,N), U\big)\rar{j}\dar[two heads]{ev_e} & \pi_1\big(\Map(S^2,N),\Emb(S^2,N), U\big) \rar{\Da}[swap]{\cong} & \A\\
        \Stab[U_*] &  &
    \end{tikzcd}
    \end{equation}
By definition the composite $\Da\circ j\circ i$  sends $\beta\in\pi_3N$ to $\Da(A)$ where $A$ is a loop in $\Map_*(S^2,N)$, based at $U$. We will see that this precisely agrees with $\phi_{[U_*]}$
from \eqref{eq-def:phi}.

Moreover, we will see that the induced action of $s\in\Stab[U_*]$ on the quotient of $\A$ by $\Da\circ j\circ i(\pi_3N)$
sends $a=\Da(H)$ to
    \[
        \Da(J_s)+s a s^{-1},
    \]
where $J_s$ is free self-homotopy of $U_*$ such that $ev_e(J_s)=s$, i.e.\ the projection of $J_s(e,-)$ to $N$ represents $s$. This action is precisely~\eqref{eq:aff}, so we recover:
\[
    p^{-1}[U] \cong \A_{U_*}
\]
as in Theorem~\ref{thm:free}, except that instead of $\fq_{U_*}$ this map is now naturally called $\Da$. This will be stated as Theorem~\ref{thm:main-Dax-part-2}.

\begin{rem}
    Using the analogous fibration sequence $\Emb_*(S^2,N)\to \Emb(S^2,N)\to N$,  we show in \cite{KST-Dax} that there is an isomorphism
\[\begin{tikzcd}
    i^{rel}\colon \pi_1(\Map_*(S^2,N),\Emb_*(S^2,N), U_*)
    \rar{\cong}
    & \pi_1(\Map(S^2,N),\Emb(S^2,N), U).
\end{tikzcd}
\] 
    In Theorem~\ref{thm:main-Dax-part-1}  we will show that  $\Da\circ i^{rel}\circ j_*=\phi_{[U_*]}$ precisely gives the indeterminacy for the based setting: $p_*^{-1}[U_*]\cong\A/\phi_{[U_*]}(\pi_3N)$, as in Theorem~\ref{thm:based-fq}.
    Moreover, this shows that $\Da\circ i^{rel}\circ j_*=\Da\circ j\circ i$.
    Therefore, similarly to our first approach, in this approach we see: a linear action for the based setting (just the quotient by the image of $\phi_{[U_*]}$), and an affine action for the free setting (the further quotient by the action \eqref{eq:aff}). 
\end{rem}


\subsection*{Acknowledgements} RS was supported by a Simons Foundation \emph{Collaboration Grant for Mathematicians}. All the authors thank the Max Planck Institute for Mathematics in Bonn, 
as well as SRS SwissMAP Research Station in Les Diablerets, for support during this project.

\tableofcontents

\section{Intersection invariants and homotopies}\label{sec:int-invariants-and-homotopies}

\subsection{3-manifolds in 6-manifolds}\label{subsec:3-in-6}

Recall that for a smooth oriented 6-manifold $P^6$, the intersection and self-intersection  invariants give maps
\[
    \lambda_3\colon \pi_3P \times \pi_3P \to  \Z\pi_1P  \quad\text{ and }\quad \mu_3\colon \pi_3P  \to  \Z\pi_1P / \langle  g+g^{-1}, 1 \rangle.
\]
To compute $\lambda_3$ geometrically, start by representing the two homotopy classes by transverse based maps $S^3 \to P$, and then count each intersection point $p$ with a sign $\epsilon_p$ determined by orientations and a group element $g_p\in\pi_1P$ represented by a sheet-changing loop through $p$.
Here a \emph{based map} is equipped with a \emph{whisker}, which is an arc running between a basepoint on the image of the map and the basepoint of the ambient manifold $P$. Note that by general position a map of a manifold of codimension $>1$ is ambient isotopic to a map whose basepoint is equal to the basepoint of the ambient manifold.  

Similarly, for $\mu_3$ one represents the homotopy class by a generic map $A\colon S^3\imra P$ and counts self-intersections, again with signs and group elements. In this dimension, switching the ordering of sheets at a double point $p$ changes $\e_p$ to $-\e_p$, and changes $g_p$ to $g_p^{-1}$, explaining the relation $g+g^{-1} =0$ in the range of $\mu_3$. The relation $1=0$ makes $\mu_3(A)$ only depend on the homotopy class of $A$, since a cusp homotopy introduces a double point with arbitrary sign and trivial group element. Changing the whisker on $A$ changes $\mu_3(A)$ by a conjugation, with the corresponding group element represented by the difference of the whiskers. 
The argument for homotopy invariance of $\mu_3$ arises from considering the double-point arcs and circles of the track of a generic homotopy $S^3 \times  I\imra P^6 \times  I$ of $A$. 

Using the involution $\bar g\coloneqq g^{-1}$ on $\Z\pi_1P$, the ``quadratic form'' $(\lambda_3, \mu_3)$ satisfies the  formulas
\begin{equation}\label{eq:mu-lambda}
    \mu_3(A+B)= \mu_3(A) + \mu_3(B) +[\lambda_3(A,B)]  \quad\text{and} \quad\lambda_3(A,A) = \mu_3(A) -\overline{\mu_3(A)} \in \Z\pi_1P
\end{equation}
where the second formula has no content for the coefficient at the trivial element in $\pi_1P$: Since $\lambda_3$ is skew-hermitian, it vanishes on the left hand side, whereas it is automatically zero on the right hand side that is defined by picking a representative of $\mu_3(A)\in\Z\pi_1P$ and then applying the involution to that specific choice.

We will be interested in the case that $P=N\times I$ is the product of a $5$--manifold $N$ with an interval $I$,
and we denote the target of $\mu_3$ by 
\[
    \A\coloneqq 
    \faktor{\Z\pi_1N}{\langle g+g^{-1} , 1 \rangle}
    \cong \faktor{\Z\pi_1P}{\langle  g+g^{-1}, 1 \rangle}.
\]

\subsection{The self-intersection invariant for homotopies of 2-spheres in 5-manifolds}\label{sec:homotopies}

The above descriptions of $\lambda_3$ and $\mu_3$ can also be applied to properly immersed simply-connected $3$--manifolds with boundary in a $6$--manifold. In this setting the invariants are computed just as above, by summing signed double point group elements, and are invariant under homotopies that restrict to isotopies on the boundary. 

For a smooth oriented 5-manifold $N$, and any homotopy $H\colon S^2 \times I \to N$ between embedded spheres, we define the self-intersection invariant 
\begin{equation}\label{eq-def:mu3-H}
    \mu_3(H)\in\A
\end{equation}
to be the self-intersection invariant of a generic track $S^2 \times I \imra N \times I$ for $H$. We will sometimes use the same letter $H$ to denote either the homotopy or its track when the context is clear.

The ``time'' parameter (the $I$-factor) of a homotopy will generally be assumed to be the unit interval $I=[0,1]$, although frequently suppressed from notation and/or reparametrized without mention.

For the purpose of computing the self-intersection invariant $\mu_3(H)$, the whisker on the track of $H$ will be assumed to be taken at the ``start'' $H(S^2\times 0)\subset N\times 1\subset N\times I$ of the homotopy unless explicitly stated otherwise. So for a homotopy $H$ from $U_*\colon S^2\hra N$ to an embedding, the whisker for $U_*$ will generally be used to compute $\mu_3(H)$.

Note that  
choosing a whisker on the track of a homotopy to provide a ``basing'' for the purposes of computing an intersection invariant is different than saying that the homotopy is a ``based homotopy'', which is ``a homotopy through based maps''.

\subsection{Geometric action of \texorpdfstring{$\A$}{A}}\label{subsec:geo-action}

For $g\in \pi_1N$ and $U_*\colon S^2\hra N$, we define $g\cdot  U_*\colon S^2\hra N$ as follows, see Figure~\ref{geo-action-fig} for several examples. Firstly, note that the normal bundle of $U_*$ is 3-dimensional, so its meridian $m_F$ is a 2-sphere; we choose it over a point near the basepoint $z\coloneqq U_*(e)$ and orient it according to the orientations of $S^2$ and $N$. We then define $g\cdot  U_*$ as an ambient connected sum of $U_*$ with $m_F$ along a tube following an arc representing $g$, where the arc starts and ends near $z$ and has interior disjoint from $U_*$.
\begin{figure}[!htbp]
        \centerline{\includegraphics[width=0.9\linewidth]{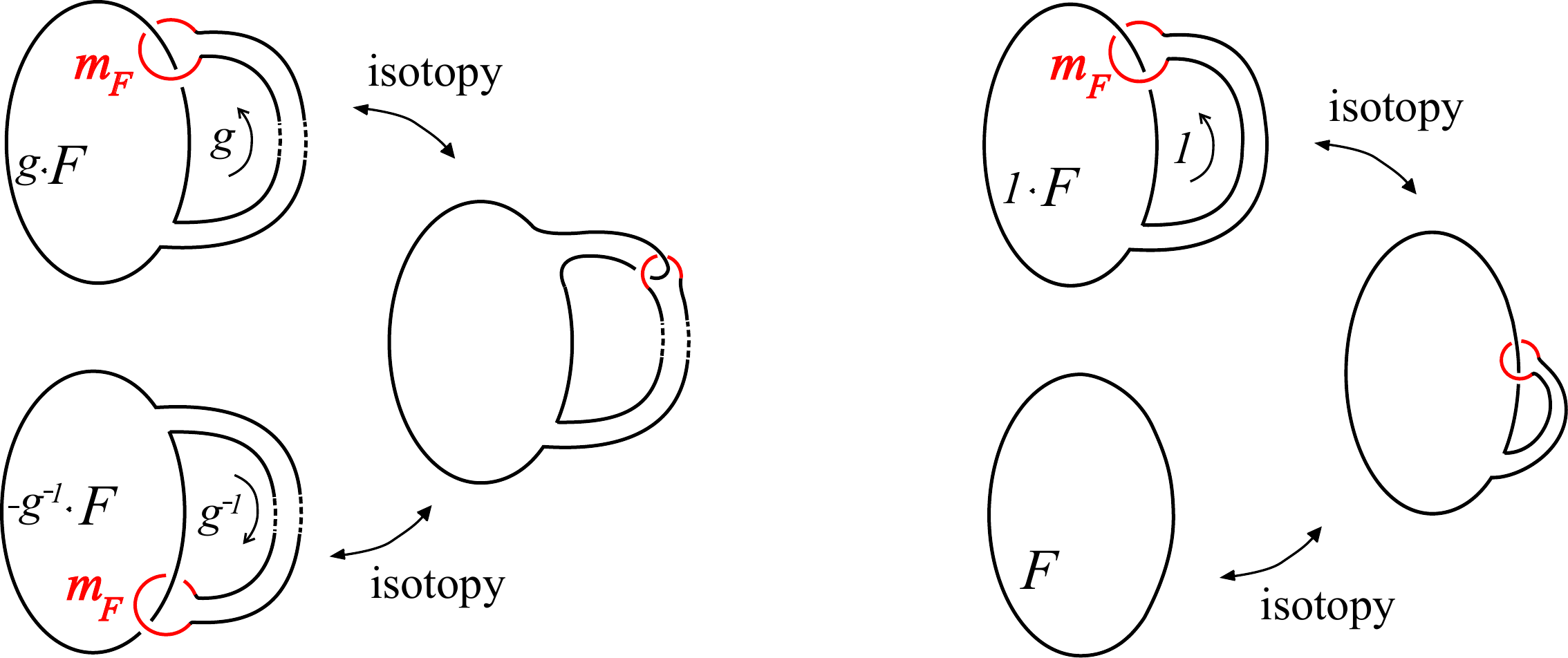}} 
        \caption{The relations $g=-g^{-1}$ and $1=0$ realized by isotopies.}
        \label{geo-action-fig}
\end{figure}

Similarly, $(-g)\cdot  U_*$ is defined to be the connected sum of $U_*$ with the oppositely oriented meridian sphere $-m_F$.
Linear combinations $\sum_i n_i g_i$ act by multiple connected sums along $g_i$ into copies of $m_F$ for $n_i>0$ respectively $-m_F$ for $n_i<0$. It is not hard to check that the relations $g+g^{-1} = 0 =1$ carry over to isotopies of these connected sums, see Figure~\ref{geo-action-fig}. Therefore, we have an action of $\A$ on the set $\pi_0\Emb_*(S^2,N)$.

Since each meridian sphere $m_F$ bounds a normal 3-ball that intersects $U_*$ exactly once, 
we get the following result.
\begin{lem}\label{lem:mu-action}
    For any $a\in\A$, shrinking the meridian spheres along their 3-balls gives a based homotopy $H_a$ from $a\cdot  U_*$ to $U_*$ with $\mu_3(H_a)=a$. 
    $\hfill\square$
\end{lem}

\subsection{The level-preserving Whitney move}\label{subsec:level-preserving-W-move}

\begin{prop}\label{prop:6d-Whitney}
    The track $H\colon S^2 \times I \to N^5\times I$ of a homotopy between two embeddings is homotopic (rel boundary) to the track of an isotopy if and only if its self-intersection invariant $\mu_3(H)\in\A$ vanishes.
\end{prop}
It will follow from the proof that if the original homotopy is a based homotopy, then the resulting isotopy can be taken to be based.
In fact, the construction given in the proof can be taken to be supported away from any $I$-family of whiskers. 

We remark that since the classical Whitney move works for immersed $3$-manifolds in $6$-manifolds \cite[Thm.6.6]{Milnor-lectures}, the vanishing of $\mu_3(H)\in\A$ immediately implies that the track $S^2 \times I \imra N \times  I$ is homotopic (rel boundary) to a concordance, so Proposition~\ref{prop:6d-Whitney} would then follow from Hudson's Theorem that concordance implies isotopy in codimensions $\geq 3$ \cite{Hudson}. Rather than invoking Hudson's result, 
our proof of Proposition~\ref{prop:6d-Whitney} will show that one can arrange for the Whitney moves to preserve $I$-levels in order to directly achieve an isotopy rather than just a concordance. 
\begin{proof}[Proof of Proposition~\ref{prop:6d-Whitney}]
    The ``only if'' direction is clear since $\mu_3$ is invariant under homotopy and vanishes on embeddings.
    
    To prove the ``if'' direction, we first introduce some streamlined notation that will only be used in the proof of Proposition~\ref{prop:6d-Whitney}, including the ancillary Lemma~\ref{lem:pairs-with-same-time-coordinate}.
    
    \textbf{Notation.}
    For any subset $\sigma\subset I$, denote by $H_\sigma\coloneqq H|_{S^2\times\sigma}$ the restriction to $S^2\times\sigma$ of the track $H\colon S^2 \times I \to N\times I$. 
    By the standard abuse of the notation, $H_\sigma\coloneqq H(S^2\times\sigma)$ is also the image of this map, and is contained in the subset $N_\sigma\coloneqq N\times \sigma\subset N\times I$.

\begin{lem}\label{lem:pairs-with-same-time-coordinate}
    For $H$ as in Proposition~\ref{prop:6d-Whitney} with $\mu_3(H)=0\in\A$, it may be arranged by a homotopy rel $\partial$ that there exist finitely many distinct points $c_i\in I$ such that the transverse self-intersections of $H$ occur in pairs $\{p_i,q_i\}\subset H_{c_i}$ with 
    $g_{p_i}=g_{q_i}$ and $\epsilon_{p_i}=-\epsilon_{q_i}$ for each $i$.
\end{lem}
    Assuming Lemma~\ref{lem:pairs-with-same-time-coordinate} (which will be proved just below), 
Proposition~\ref{prop:6d-Whitney} follows by a standard application of Whitney moves to eliminate each of the self-intersection pairs of Lemma~\ref{lem:pairs-with-same-time-coordinate} in a way that yields the track of an isotopy. We describe details here for completeness, with the key observation being that each Whitney disk can be chosen to be contained in a level.
    
    Dropping the subscript $i$ from the notation, let $p$ and $q$ be a pair of self-intersections of $H_c$ as in Lemma~\ref{lem:pairs-with-same-time-coordinate}. Since $H_c$ is a map of a $2$-sphere $S^2\times c$, the self-intersections $p$ and $q$ are not transverse for $H_c$; but there exists some small $\delta>0$ such that $\{p,q\}=H_{[c-\delta,c+\delta]}\pitchfork H_{[c-\delta,c+\delta]}$, the transverse self-intersections of the immersed $3$-manifold $H_{[c-\delta,c+\delta]}$ in the $6$-manifold $N_{[c-\delta,c+\delta]}$.
    
    Since $p$ and $q$ have the same group elements $g_p=g_q$ and opposite signs $\epsilon_p=-\epsilon_q$, there exists a \emph{Whitney disk} $W\subset N_{[c-\delta,c+\delta]}$ pairing $p$ and $q$. By general position we may assume that $W$ is embedded in the $5$-dimensional slice $N_c\subset N_{[c-\delta,c+\delta]}$ with interior disjoint from $H_c$. 
    The Whitney disk boundary $\partial W=\alpha\cup\beta$ is the union of embedded arcs $\alpha$ and $\beta$ contained in $H_c$, with $\alpha\cap\beta=\{p,q\}$.
    Let $\overline{\alpha}$ and $\overline{\beta}$ be slightly longer arcs in $H_c$ containing $\alpha$ and $\beta$, respectively, such that $\overline{\alpha}$ and $\overline{\beta}$ extend just beyond $p$ and $q$. 

    Let $A,B\subset H_{[c-\delta,c+\delta]}$ denote regular $3$-ball neighborhoods of 
    $\overline{\alpha}$ and $\overline{\beta}$ in $H_{[c-\delta,c+\delta]}$. 
    Each of $A$ and $B$ is ``the image of a local sheet of a $2$-sphere $H_t$ moving in time'', with $A_t$ and $B_t$ each embedded in $H_t$ for $t\in [c-\delta,c+\delta]$, and $A\cap B=\{p,q\}$. 
    It follows that $A:D^2\times I\hra N\times I$ is the track of an isotopy $A_t$.

    The Whitney move that eliminates $p$ and $q$ will be described using a particular choice of coordinates for 
    an open neighborhood $V\subset N_{[c-\delta,c+\delta]}$ containing $W$.
    
    By \cite[Lem.6.7]{Milnor-lectures}, $V$ may be chosen to be diffeomorphic to 
    $\overline{W}\times\R^2_A\times\R^2_B$, where 
    \begin{itemize}
        \item 
        $\overline{W}\subset N_c$ is a smooth $2$-disk 
        formed from $W$ by attaching a half-open collar to $\partial W$,
        
        \item
        $V\cap A=\overline{\alpha}\times\R^2_A\times (0,0)$, 
        
        \item
        $V\cap B=\overline{\beta}\times(0,0)\times\R^2_B.$ 
    
    \end{itemize}
    Let $\overline{\alpha}(s)$ be a smooth isotopy of the arc $\overline{\alpha}$ in $\overline{W}$, for $0\leq s\leq 1$, such that $\overline{\alpha}(0)=\overline{\alpha}$, and $\overline{\alpha}(1)$ passes just above $\beta\subset \overline{\beta}$ as in Figure~\ref{alpha-isotopy-fig}.
    In particular, $\overline{\alpha}(s)$ is supported near $W\subset \overline{W}$ for all $0\leq s\leq 1$, and $\overline{\alpha}(1)$ is disjoint from $B$.
    \begin{figure}[!htbp]
            \centerline{\includegraphics[width=0.7\linewidth]{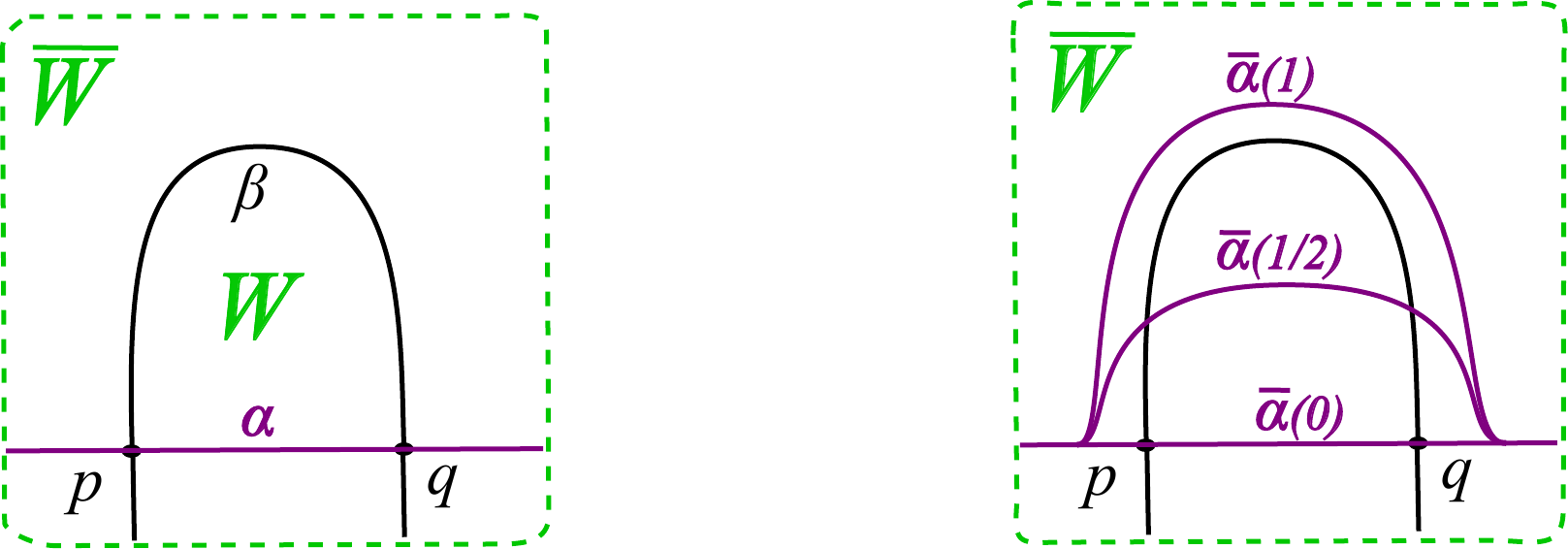}}
            \caption{}\label{alpha-isotopy-fig}
    \end{figure}
    
    Let $\rho\colon \R^2\to[0,1]$ be a smooth bump function $(u,v)\mapsto \rho(u,v)$ such that:
    \begin{itemize}
        
        \item 
        $\rho(u,v)=1$ if $\sqrt{u^2+v^2}\leq 1$,
        
        \item
        $\rho(u,v)=0$ if $\sqrt{u^2+v^2}\geq 2$. 
        
    \end{itemize}
    
    Now we use $\overline{\alpha}(s)$ and $\rho$ to define the Whitney move as the result of an isotopy $A(s)$, $0\leq s \leq 1$ of $A$ which fixes $B$:
    \[
    A(s)= \overline{\alpha}(s\rho(a_1,a_2))\times(a_1,a_2)\times(0,0)\subset \overline{W}\times\R^2_A\times\R^2_B,
    \]
    where $(a_1,a_2)$ runs through $\R^2_A$.
    Then $A(0)=A$, and the result of the $W$-Whitney move on $A$ is $A':=A(1)$, so that $A'\cap B=\emptyset$.

    Note that $A(s)=A$ near $\partial V$, and hence we can extend $A(s)$ to be the identity outside $V$.
    This defines a homotopy $H(s)$ of $H=H(0)$ such that $H'\coloneqq H(1)$ satisfies $H'\pitchfork H'=H\pitchfork H - \{p,q\}$. 
    
    By construction $A(s)$ only moves points of $A$ along the $\overline{W}$-factor, which is orthogonal to the $I$-factor of $N\times I$ since $\overline{W}\subset N_c$. 
    This means that each $A(s)$ consists of the track of isotopies $A_t(s)$, and similarly for $H(s)$.

    Performing Whitney moves on all the self-intersection pairs of Lemma~\ref{lem:pairs-with-same-time-coordinate} yields the track of an isotopy as desired.
\end{proof}

\begin{proof}[Proof of Lemma~\ref{lem:pairs-with-same-time-coordinate}]
    The condition $\mu_3H=0\in\A$ means that the (finite) set of transverse self-intersections of the generic track $H\colon S^2\times I\imra N\times I$
    can be decomposed into finitely many pairs $\{p_i,q_i\}$ with 
    $g_{p_i}=g_{q_i}$ and $\epsilon_{p_i}=-\epsilon_{q_i}$ (for appropriately chosen sheets, and after perhaps performing a single cusp homotopy on $H$). 
    
    Suppose that for some $i$, we have $p_i\in H_{t_1}\cap H_{t_1}\subset H\pitchfork H$ and $q_i\in H_{t_2}\cap H_{t_2}\subset H\pitchfork H$, with $t_1<t_2$.
    We will describe how to change $H$ by an isotopy rel boundary which ``moves'' $q_i$ to $q_i'\in H_{t_1}\cap H_{t_1}$ and is supported away from all other self-intersections of $H$. The construction will show more generally that self-intersections of $H$ can be arranged to occur at any chosen times, while preserving signs and group elements.  
    
    \emph{Special case.}
    First consider the special case that $H$ has just a single pair $\{p,q\}=H\pitchfork H$ of self-intersections with $p\in H_{t_1}\cap H_{t_1}$ and $q\in H_{t_2}\cap H_{t_2}$, and with $0<t_1<t_2<1$ in $I=[0,1]$.
    Let $(x,t_2)$ and $(y,t_2)$ be the two distinct preimages in $S^2\times t_2\subset S^2\times I$ of $q= H_{t_2}((x,t_2))=H_{t_2}((y,t_2))$.
    Define vertical arcs $a \coloneqq x\times [t_1,t_2]$ and $b\coloneqq y\times [t_1,t_2]$ in the domain $S^2\times I$ (see the left side of Figure~\ref{nested-bump-functions-pair-fig}).   

\begin{figure}[!htbp]
        \centerline{\includegraphics[scale=.35]{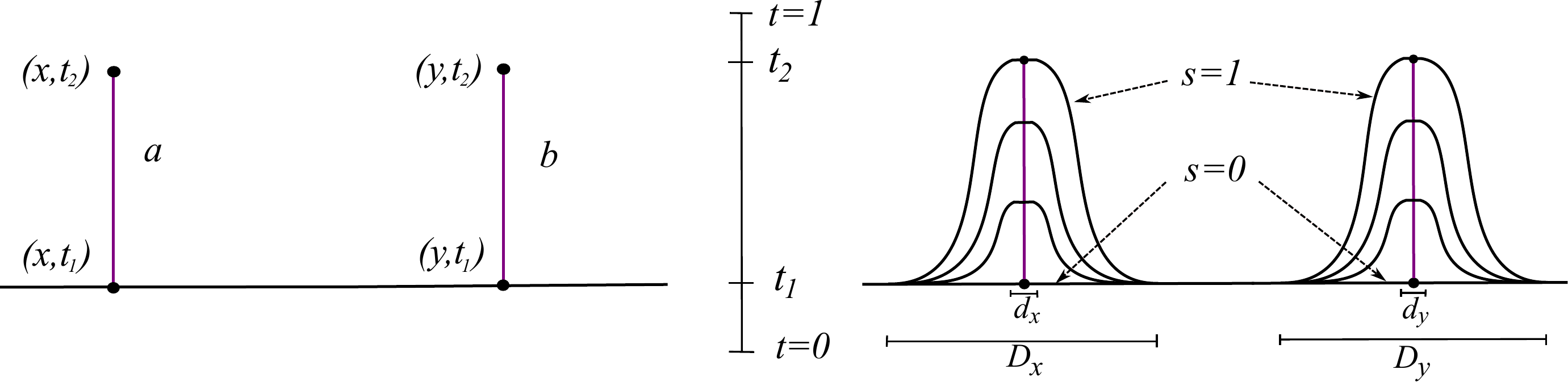}}
        \caption{Schematic pictures in the domain $S^2\times I$, with the $I$-factor running vertically from bottom to top. Left: The vertical arcs $a$ and $b$, and a horizontal sub-arc of $S^2\times t_1$. Right: Images of the sub-arc of $S^2\times t_1$ containing the nested disks $d_x\subset D_x$ and $d_y\subset D_y$ under the isotopy $\psi_s$ for $s=0$, $s=1$, and for two other intermediate values of $s$.}\label{nested-bump-functions-pair-fig}
\end{figure}
    For $0\leq s\leq 1$ we will define a smooth isotopy $\psi_s\colon S^2\times I\to S^2\times I$ supported near $a\cup b$ such that $H\circ\psi_s$ satisfies
    $H=H\circ\psi_0$ and 
    $(H\circ\psi_1)\pitchfork (H\circ\psi_1)=\{p,q'\}\subset (H\circ\psi_1)_{t_1}$. 
    
    First we define $\psi_s$ on $S^2\times t_1$ as the sum of two local bump functions of height $s(t_2-t_1)$ in the positive $I$-direction centered at $(x,t_1)$ and $(y,t_1)$ (see the right side of Figure~\ref{nested-bump-functions-pair-fig}). 
    More specifically, let $x\in d_x\subset D_x\subset S^2$ and $y\in d_y\subset D_y\subset S^2$ be small concentric pairs of nested disks around $x$ and $y$, respectively. Then
    \[
    \psi_s(z,t_1) = 
    \begin{cases}
     (z,t_1)  &  \text{ if } z\notin D_x\cup D_y \\
     (z,t_1+s(t_2-t_1)) & \text{ if } z\in d_x\cup d_y\\
     (z,t_1+\sig(z)s(t_2-t_1)  & \text{ if } z\in (D_x\setminus \int(d_x))\cup (D_y\setminus \int(d_y))
    \end{cases}
    \]
    Here the sigmoid function $\sig(z)$ smoothly interpolates between $\sig(z)=0$ for $z\in\partial D_x\cup \partial D_y$ and $\sig(z)=1$ for $z\in \partial d_x\cup \partial d_y$.
    
    Now extend $\psi_s$ to all of $S^2\times I$ by tapering the bump functions down to zero as $t$ moves away from $t_1$, so that $\psi_s(z,0)=(z,0)$ and $\psi_s(z,1)=(z,1)$ for all $s$. See Figure~\ref{nested-bump-functions-extended-fig} for an illustration of the extended $\psi_1\colon S^2\times I\to S^2\times I$.
    \begin{figure}[!htbp]
            \centerline{\includegraphics[scale=.35]{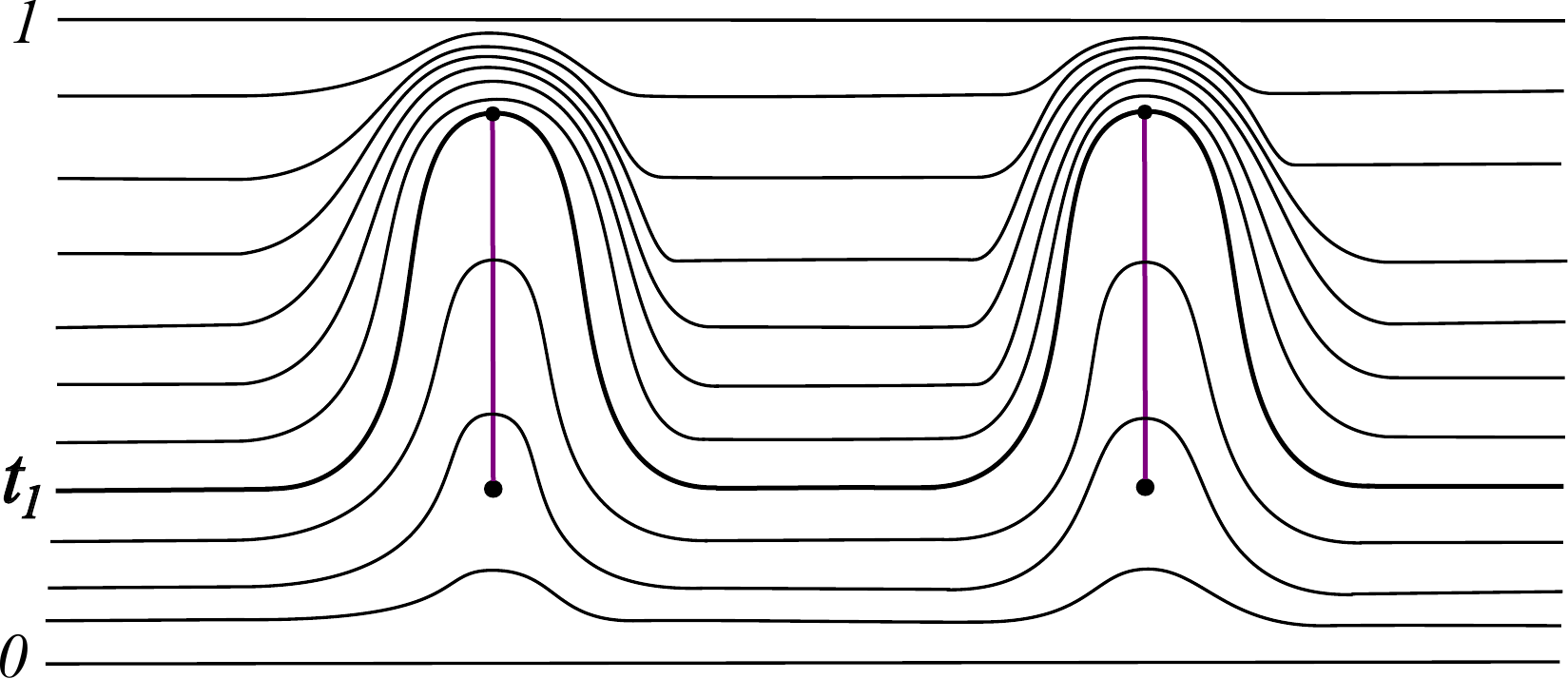}}
            \caption{A schematic picture of the images under $\psi_1$ of some horizontal slices of $S^2\times I$.}\label{nested-bump-functions-extended-fig}
    \end{figure}

    Next we check that $H'\coloneqq H\circ\psi_1$ has the desired properties. First, observe that since $\psi_s$ restricts to the identity map on the complement of $(D_x\cup D_y)\times I$, we have $H'=H$ when restricted to 
    $(S^2\setminus (D_x\cup D_y))\times I$. In particular, $H'((S^2\setminus (D_x\cup D_y))\times I)$ has only the single transverse self-intersection point $p\in H\pitchfork H$ which occurs at $t_1$. 
    
    Now consider the restriction of $H'$ to $(D_x\cup D_y)\times I$. 
    By construction, $H'((D_x\cup D_y)\times I)$ has just the single transverse self-intersection $q'=H'(d_x,t_1)\cap H'(d_y,t_1)$, which is the image of the projection to $N_{t_1}$ of $q\in H\pitchfork H$.
    
    It remains to check that there are no transverse intersections between  $H'((D_x\cup D_y)\times I)$ and $H'((S^2\setminus (D_x\cup D_y))\times I)$.
    For each $t\in I$, let $\proj_t\colon N\times I\to N_t$ be the projection map.
    By general position, for each $t\in I$ the image $\proj_t\circ H(a\cup b)\subset N_t$ is an embedded arc $\gamma_t$. 
    Since the image $H'((D_x\cup D_y)\times I)$ in the $6$-manifold $N\times I$ is contained near the $2$-dimensional union $\cup_{t\in I}\,\gamma_t$, it follows by general position that $H'((D_x\cup D_y)\times I)$ has no transverse intersections with
    the $3$-dimensional $H'((S^2\setminus (D_x\cup D_y))\times I)$.
    
    \emph{General case.}
    Since by general position any number of self-intersections can be assumed to have preimages projecting to distinct points in $S^2$,
    the above construction moving $q\in H_{t_2}\cap H_{t_2}$ to $q'\in H'_{t_1}\cap H'_{t_1}$ for $t_1<t_2$ can be carried out iteratively (or even simultaneously) for any chosen subset of self-intersections while fixing the complementary subset.
\end{proof}

\subsection{Based self-homotopies}\label{subsec:based-self-htpies}
Recall from section~\ref{subsec:intro-based-isotopy} of the introduction that, for a fixed based embedding $U_*\colon S^2 \hra N$, we denote by $\A_{[U_*]}$ the quotient of $\A$ by the image of the indeterminacy homomorphism 
$\phi_{[U_*]}\colon \pi_3N \longrightarrow \A$ defined by $A\mapsto\mu_3(A) + [\lambda_N(A,[U_*])]$.

The following lemma will be used to show that our invariants are well defined:
\begin{lem}\label{lem:mu-of-based-self-homotopies}
    If $J_*\colon S^2 \times I \imra N \times I$ is a generic track of a based self-homotopy of $U_*$,
    then
    \[
        \mu_3(J_*)=0 \in \A_{[U_*]}.
    \]
\end{lem}
\begin{proof}
    Since $J_*$ is a based self-homotopy, it agrees with $U_* \times I$ on the 2-skeleton $S^2 \times \{0,1\} \cup e \times I$ of $S^2 \times I$, with $e\in S^2$ the basepoint. 
    So they only differ on the 3-cell, where $U_* \times I$ is represented by $B_U\coloneqq  U_*(D^2)\times I$ (here $D^2$ is the complement in $S^2$ of a small disk around $e$) and
    $J_*$ is represented by a generic $3$--ball $B_J\colon D^3\imra (N \times I) \smallsetminus \nu( U_*(e) \times I)$.
    By construction, the boundaries of these two 3-balls are parallel copies of an embedded 2--sphere in the boundary of a small neighborhood of $U_* \times \{0,1\}\cup( U_*(e) \times I)$.
    
    Gluing $B_U$ and $B_J$ together along a small embedded cylinder $S^2 \times I$ between their boundaries yields a map of a $3$--sphere
    $A\coloneqq B_J\cup (-B_U)\colon S^3 \to N \times I$.
    To prove the lemma we will show that $\mu_3(J_*)=\phi_{[U_*]}(A)$.
    
    First note that on one hand, all contributions to $\mu_3(J_*)$ come from
    the self-intersections of the immersed $3$--ball $B_J$.
    On the other hand, contributions to $\mu_3(A)$ come from the self-intersections of $B_J$ and the intersections between $B_J$ and the embedded $3$--ball $-B_U$. The latter intersections are precisely counted by $-\lambda_3(J_*, U_* \times I)$, cf.\ \eqref{eq:mu-lambda}.
    Therefore,
    \[
        \mu_3(J_*) - \mu_3(A) = \lambda_3(J_*, U_* \times I),
    \]
    and since $\lambda_N(A,  U_*) = \lambda_3(A,  U_* \times I) = \lambda_3(J_*, U_* \times I)$,
    we obtain
    \[
        \mu_3(J_*) = \mu_3(A) + \lambda_N(A,  U_*)=\phi_{[U_*]}(A).\qedhere
    \]
\end{proof}

\subsection{From homotopy to isotopy by adding 3-spheres}\label{subsec:htpy-to-istpy}
The following lemma will be used to show that our invariants are injective. 
\begin{lem}\label{lem:homotopies-to-isotopies}
Suppose $H\colon S^2 \times I \imra N \times I$ is a generic track of a homotopy between embeddings $R\colon S^2\hra N$
and $R'\colon S^2\hra N$ such that the homotopy restricts to an embedding $U\colon S^2\hra N$ for some point in $I$.
Then $R$ is isotopic to $R'$ if $\mu_3(H)=0\in\A_{[U_*]}$.
Moreover, if $H$ is a based homotopy, then the resulting isotopy may be taken to be based.
\end{lem}

\begin{proof}
Since $\mu_3(H)=0\in\A_{[U_*]}$, there exists $A\in\pi_3N$ such that $\mu_3(H)=\phi_{[U_*]}(A)$.
If $U$ is the restriction of $H$ to $t_0\in I$, then by a small ambient isotopy we may assume that $H$ restricts to a product 
$U\times [t_0-\epsilon,t_0+\epsilon]$ on a small interval around $t_0$.
Represent $A$ by a generic regular homotopy $f_t\colon S^2 \times [t_0-\epsilon,t_0+\epsilon]\to N$ from the trivial sphere $f_{t_0-\epsilon}=f_{t_0+\epsilon}$ in $N$ to itself.
Taking a smooth family of ambient connected sums of $f_t$ with $U\times t\subset  U\times [t_0-\epsilon,t_0+\epsilon]$ yields a self-homotopy $J^A$ of $U$. We can assume the guiding paths for these connected sums have interiors disjoint from every $f_t$ and $U\times t$, so that $\mu_3(J^A)=\phi_{[U_*]}(A)$.

Inserting $-J^A$ into $H\times [t_0-\epsilon,t_0+\epsilon]\subset H\times I$ yields a based homotopy $H^0$ between $R$ and $R'$ with $\mu_3(H^0)=0\in\A$, and $H^0$ is homotopic rel boundary to an isotopy from 
 $R$ to $R'$ by Proposition~\ref{prop:6d-Whitney}. If $H$ is a based homotopy, then this resulting isotopy inherits the extended whiskers from $H$.
\end{proof}

\section{Homotopy versus isotopy}
In section~\ref{sec:based-setting} we recall the statement of Theorem~\ref{thm:based-fq} describing the based setting, and give a proof.
In section~\ref{sec:free-setting} we clarify and prove Theorem~\ref{thm:free}, describing the free setting.

Our convention is to write concatenations of homotopies as unions from left to right, with a minus sign indicating that the orientation of the $I$-factor has been reversed.
Recall (section~\ref{sec:homotopies}) that for the purposes of computing the self-intersection invariant $\mu_3(H)$ of a homotopy $H$ the whisker on the track of $H$ will be assumed to be taken at the ``start'' $H(S^2\times 0)\subset N\times 1\subset N\times I$
unless otherwise explicitly specified.

\subsection{The based setting}\label{sec:based-setting}

Theorem~\ref{thm:based-fq} states that for a fixed based embedding $U_*\colon S^2 \hra N^5$ the map
\[
    p_*^{-1}[U_*]\ra \A_{[U_*]},\quad R_* \mapsto \fq_{[U_*]}(R_*)\coloneqq [\mu_3(H_*)],
\]
for $H_*$ any based homotopy from $U_*$ to $R_*$, is a bijection.

Here $p_*^{-1}[U_*]$ is the set of based isotopy classes of embedded spheres $R_*\colon S^2\hra N^5$ that are based homotopic to $U_*$. Moreover, the group $\A_{[U_*]}$ is the quotient of $\A\coloneqq \Z\pi_1N/ \langle g+g^{-1} , 1 \rangle$ by the image of the indeterminacy homomorphism 
$\phi_{[U_*]}\colon \pi_3N \longrightarrow \A$ defined by 
\[
    A\mapsto\mu_3(A) + [\lambda_N(A,[U_*])].
\]

\subsubsection{\texorpdfstring{$\fq_{[U_*]}$}{fq} is well defined}\label{subsec:proof-based-fq-well-defined}

It suffices to show that $\fq_{[U_*]}(R_*)\in\A_{[U_*]}$ is independent of the choice of $H_*$.

Taking the union along $R_*$ of any two based homotopies $H_*, H'_*$ from $U_*$ to $ R_*$ gives a based self-homotopy $J_*=H_*\cup_{R_*} -H_*'$ of $U_*$ such that 
\[\mu_3(J_*)=\mu_3(H_*\cup_{R_*} -H_*') = \mu_3(H_*) - \mu_3(H_*')\]
Since $\mu_3(J_*)$ lies in the image of $\phi_{[U_*]}$ by Lemma~\ref{lem:mu-of-based-self-homotopies},
we have $[\mu_3(H_*)]=[\mu_3(H_*')]\in\A_{[U_*]}$.

\subsubsection{\texorpdfstring{$\fq_{[U_*]}$}{fq} is injective}\label{subsec:proof-based-fq-injective}

If $\fq_{[U_*]}(R_*)=\fq_{[U_*]}(R'_*)$, then there exist based homotopies $H_*$ and $H'_*$ from $U_*$ to $R_*$ and $R'_*$, respectively, such that $\mu_3(H_*)=\mu_3(H'_*)\in\A_{[U_*]}$. 
Taking the union of these homotopies along $U_*$ gives a based homotopy $H''_*\coloneqq H_*\cup_{U_*}-H'_*$ from $R_*$ to $R'_*$ with $\mu_3(H''_*)= \mu_3(H_*)-\mu_3(H'_*)=0\in\A_{[U_*]}$.
It follows from Lemma~\ref{lem:homotopies-to-isotopies} that $R_*$ is based isotopic to $R'_*$.

\subsubsection{\texorpdfstring{$\fq_{[U_*]}$}{fq} is surjective}\label{subsec:proof-based-fq-surjective}
Surjectivity follows directly from Lemma~\ref{lem:mu-action}.

\subsection{The free setting}\label{sec:free-setting}
This section clarifies the target of the invariants in the free setting, and proves Theorem~\ref{thm:free}, which we recall here for the reader's convenience:
For $U_*$ a fixed basing of an embedding $U\colon S^2\hra N^5$, the map 
\[
    p^{-1}[U]\ra \A_{U_*},\quad R \mapsto \fq_{U_*}(R)\coloneqq [\mu_3(H)],
\]
where $H$ is any free homotopy from $U$ to $ R$, is
a bijection.

Here $p^{-1}[U]$ is the set of isotopy classes of embedded spheres $R\colon S^2\hra N^5$ that are freely homotopic to $U$. Moreover, the group $\A_{U_*}$ is the quotient set of the based target $\A_{[U_*]}$ of Theorem~\ref{thm:based-fq} by the affine action of $\Stab[U_*]<\pi_1N$ given by ${}^sa= U_s + sas^{-1}$ for all $a\in\A_{[U_*]}$ and $s\in\Stab[U_*]$, with the definition of $U_s\in\A_{[U_*]}$ given 
just after Lemma~\ref{lem:J_s}
in the next subsection.

\subsubsection{The affine action of \texorpdfstring{$\Stab[U_*]$}{Stab[U]} on \texorpdfstring{$\A_{[U_*]}$}{AF}}\label{subsec:stabilizer-action} 
For each $s\in\Stab[U_*]$ there is a track 
\[
    J_s\colon S^2 \times I \to N\times I
\]
of a free self-homotopy of $U_*$ such that the projection of $J_s(e,-)$ represents $s$, where the basepoint $e\in S^2$ is the preimage of the basepoint of $U_*$.
We say that $J_s$ \emph{represents} $s\in\Stab[U_*]$, and call $s$ the \emph{core} of $J_s$, frequently using the subscript notation to indicate this representation.

Note the following three properties of core elements:
\begin{enumerate}
\item 
Any self-homotopy whose core is the trivial element of $\pi_1N$ is homotopic rel boundary to a based self-homotopy.

\item
Concatenating self-homotopies multiplies the core elements: $J_{sr}=J_s\cup J_r$. 

\item
Reversing a self-homotopy inverts its core: $-J_s=J_{s^{-1}}$. 

\end{enumerate}

It follows from these three properties that given two free self-homotopies $J_s$ and $J'_s$ of $U_*$ representing the same element $s\in\Stab[U_*]$,
we can form a based self-homotopy $J_1=J_{ss^{-1}}\coloneqq J_s\cup-J'_s$ of $U_*$ representing the trivial element $1\in \Stab[U_*]$, with $\mu_3(J_1)=\mu_3(J_s)-\mu_3(J'_s)$.
Together with Lemma~\ref{lem:mu-of-based-self-homotopies} we immediately get:
\begin{lem}\label{lem:J_s}
If $J_s$ and $J'_s$ are two free self-homotopies of $U_*$ representing the same element $s\in\Stab[U_*]$,
then $\mu_3(J_s)-\mu_3(J'_s)=0\in\A_{[U_*]}$.
$\hfill\square$
\end{lem}
As a result of Lemma~\ref{lem:J_s}, the element
\[
U_s \coloneqq [\mu_3(J_s)]\in\A_{[U_*]}
\]
is well defined, and hence so is the affine action ${}^sa\coloneqq U_s + sas^{-1}$ of $\Stab[U_*]$ on $\A_{[U_*]}$.
This clarifies Definition~\ref{def:unbased-A_F} of the target of the free isotopy invariant $\fq_{U_*}\in\A_{U_*}$ as the quotient set of the based isotopy target $\fq_{ [U_*]}\in\A_{[U_*]}$ under this action.

The following lemma will be used in sections~\ref{subsec:free-fq-well-defined} and~\ref{sec:dax-proofs}; it illustrates how the affine action describes the effect of free self-homotopies on the self-intersection invariant.
\begin{lem}\label{lem:basic-concatenation}
    If $H$ is a homotopy from $U_*$ to an embedding $R$, and $J_s$ is a free self-homotopy of $U_*$ representing $s\in\Stab[U_*]$, then
     the free homotopy $J_s\cup H$ from $U$ to $R$ satisfies 
    \[
    \mu_3(J_s\cup H)=\mu_3(J_s)+s\mu_3(H)s^{-1}\in\A_{[U_*]}.
    \]
\end{lem}
\begin{proof}
    It is clear that each double point of the track of $J_s\cup H$ is either a double point of $J_s$ or $H$. By our convention, the computation of $\mu_3(J_s\cup H)$ uses the whisker for $U_*$ at the start of $J_s$. Thus, double point loops of $H$ get conjugated by representatives of $s$ while traversing $J_s$, so all the double-point group elements of $\mu_3(H)$ get conjugated by $s$. 
\end{proof}

\subsubsection{\texorpdfstring{$\fq_{U_*}$}{fq} is well defined}\label{subsec:free-fq-well-defined}
It suffices to show that $\fq_{U_*}(R)=[\mu_3(H)]\in\A_{U_*}$ is independent of the choice of $H$.
For $H$ and $H'$ two choices of free homotopies from $U_*$ to $R$, the concatenation $J_s\coloneqq H\cup -H'$ is a self-homotopy of $U_*$ representing some $s\in\Stab[U_*]$, and by Lemma~\ref{lem:basic-concatenation} we have 
\[
    \mu_3(J_s)=\mu_3(H)-s\mu_3(H')s^{-1}\in\A.
\]
So $\mu_3(H)=\mu_3(J_s)+s\mu_3(H')s^{-1}= {}^s(\mu_3(H'))\in\A_{[U_*]}$ which implies $[\mu_3(H)]=[\mu_3(H')]\in\A_{U_*}$, and hence $\fq_{U_*}(R)$ is well defined.

\subsubsection{\texorpdfstring{$\fq_{U_*}$}{fq} is injective}
If $\fq_{U_*}(R)=\fq_{U_*}(R')$,
then by the definition of the target $\A_{U_*}$ there exist homotopies $H$ and $H'$ from $U_*$ to $R$ and $R'$, respectively, such that 
\[
\mu_3(H')=U_s+s\mu_3(H)s^{-1}\in\A_{[U_*]}
\] 
for some $s\in\Stab[U_*]$.

Consider the homotopy $H''\coloneqq -H'\cup J_s\cup H$ from $R'$ to $R$, where $J_s$ is any self-homotopy of $U_*$ representing $s$.
Using the whisker on $U_*$ in $-H'\cap J_s\subset H''$ we have the following computation in $\A_{[U_*]}$:

\[
\begin{array}{rcl}
\mu_3(H'')  & =  &  \mu_3(-H')+\mu_3(J_s)+s\mu_3(H)s^{-1} \\
& =  &  -\mu_3(H')+U_s+s\mu_3(H)s^{-1} \\
  &  = & -(U_s+s\mu_3(H)s^{-1})+U_s+s\mu_3(H)s^{-1}  \\
  & =  & 0 
\end{array}
\]

It follows from Lemma~\ref{lem:homotopies-to-isotopies} that $R$ is isotopic to $R'$.

\subsubsection{\texorpdfstring{$\fq_{U_*}$}{fq} is surjective}
Surjectivity follows directly from Lemma~\ref{lem:mu-action}.

\subsection{Examples}\label{sec:examples}
Recall that Corollary~\ref{cor:null-homotopic} of the Introduction states that free isotopy classes of null-homotopic 2-spheres in $N$ are in bijection with 
$
\A_{U_*}=\A_{[U_*]}  / \pi_1N 
$,
where the action is by conjugation.

Here we examine some examples of free isotopy classes of essential 2-spheres:

 \subsubsection{}

Consider $U_*=S^2\times\{p\}\subset N=S^2\times M^3$.

Then $U_s=0$ for all $s \in\Stab[U_*]= \pi_1N$, because any self-homotopy $J_s$ can be chosen to be a self-isotopy which moves $p$ around a loop representing $s$ while fixing the $S^2$-factor.

So the affine action has trivial translations (cf.~section~\ref{intro:free-isotopy}) and free isotopy classes of spheres homotopic to $U_*$ are in bijection with 
$\A_{U_*}=\A_{[U_*]}  / \pi_1M  \text{ (with conjugation action)}$.

\subsubsection{}\label{subsubsec:example-free-target-depends-on-isotopy-class}

Consider again $U_*=S^2\times\{p\}\subset N=S^2\times M^3$.

Assume $[g,h]\neq 1\in\pi_1M\cong\pi_1N$.

If $U_*^g$ is the result of doing a $g$-finger move on $U_*$, then 
$U^g_s=g-sgs^{-1}$ for each
$s\in\Stab[U_*^g]= \pi_1N$.
Here $U^g_s=\mu_3(J_s)$ where $J_s$ is a self-homotopy of $U_*^g$ that undoes the $g$-finger move, then moves $p$ around a loop representing $s$ while fixing $S^2$, and then redoes the $g$-finger move.
In particular, $U^g_h=g-hgh^{-1}\neq 0\in\A_{[U_*^g]}=\A_{[U_*]}$ if $\pi_3M=0$.

So in this case the affine action ${}^ha=(g-hgh^{-1})+hah^{-1}$ defining $\A_{U_*^g}$ as a quotient of $\A_{[U_*^g]}=\A_{[U_*]}$ has non-trivial translations, 
illustrating how the target of the free isotopy invariant depends in general on the isotopy class of the fixed embedding, and not just on its homotopy class.

This suggests the following questions:
When does a homotopy class of 2-spheres in $N$ contain an isotopy class such that 
the corresponding affine action has trivial translations?
Are stabilizers of elements of $\pi_2N$ always represented by some embedded $S^2\times S^1\subset N\times S^1$?

\section{A space level approach following Dax}
\label{sec:Dax-approach}

In this section we reprove our two main results, Theorem~\ref{thm:based-fq} and Theorem~\ref{thm:free}, using a space level approach of Dax and \cite{KST-Dax}.

\subsection{The relative homotopy group}\label{subsec:dax-preliminaries}
In \cite{KST-Dax}, following \cite{KT-highd} and~\cite{Dax}, we compute the relative homotopy group $\pi_{d-2\ell}(\Map(V,X),\Emb(V,X), U)$ for any $\ell$-manifold $V$ and $d$-manifold $X$, and a fixed embedding $U\colon V\hra X$. In our case of interest, $V=S^2$ and $X=N$ of dimension $d=5$, the relevant result is as follows.
\begin{thm}[{\cite{KST-Dax}}]
\label{thm:main-Dax-part}
    Let $N$ be an oriented connected $5$-manifold and $U_*\colon S^2\hra N$ a smooth based embedding. Then there are bijections
    \[ \begin{tikzcd}
        \Da\colon
        \pi_1\big(\Map_*(S^2,N),\Emb_*(S^2,N), U_*\big)\rar{i^{rel}}[swap]{\cong} &
        \pi_1\big(\Map(S^2,N),\Emb(S^2,N), U\big) 
        \rar[swap]{\cong} & \A,
    \end{tikzcd}
    \]
    given on a class $[H]$ as the sum over double points of the associated group elements of the track of $H\colon I\to\Map(S^2,N)$, defined by
    $I\times S^2\to I\times N$, $(t,v)\mapsto(t,H_t(v))$. In other words, $\Da([H])=\mu_3(H)$ is precisely the self-intersection invariant from \eqref{eq-def:mu3-H}.
\end{thm}

\begin{rem}
    Using Lemma~\ref{lem:mu-action} one can define an explicit inverse of $\Da$. This is completely analogous to the realization map $\realmap$ in \cite{KT-highd,KST-Dax}.
\end{rem}
Our main square from section~\ref{subsec:intro-square-diagram} extends to a commutative diagram:
\begin{equation}\label{big-diagram}\begin{tikzcd}[column sep=15pt]    
    \pi_1(\Map_*(S^2,N),U_*)\rar{j_*}\dar{i} &  \pi_1(\Map_*,\Emb_*, U_*)\rar\dar{i^{rel}} &
    \pi_0\Emb_*(S^2,N)\rar[two heads]{p_*}\dar[two heads] &
    \pi_0\Map_*(S^2,N)\dar[two heads]\\
    \pi_1(\Map(S^2,N),U)\rar{j} &  \pi_1(\Map,\Emb, U)\rar &
    \pi_0\Emb(S^2,N)\rar[two heads]{p} &
    \pi_0\Map(S^2,N)
    \end{tikzcd}
\end{equation}
The top row is a final part of the long exact sequence of the pair in the based case, $(\Map_*,\Emb_*)\coloneqq(\Map_*(S^2,N),\Emb_*(S^2,N))$, whereas the bottom row is from the long exact sequence of the pair in the corresponding free case.

We use the following standard facts about homotopy groups of mapping spaces, see~\cite{KST-Dax} for details.
\begin{lem}\label{lem:htpy-grps}
   There are isomorphisms $\pi_k(\Map_*(S^2,N), U_*)\to\pi_{k+2}(N)$, 
   for $k\geq0$, and a bijection $\pi_0\Map(S^2,N)\cong\pi_2N/\{\alpha - g\alpha\}$ for the usual action of $g\in\pi_1N$ on $\alpha\in\pi_2N$. For any $\beta\in\Map_*(S^2,N)$ there is an exact sequence
    \[
    \begin{tikzcd}
        \pi_2N\rar & \pi_3N\cong \pi_1(\Map_*(S^2,N),\beta) \rar{ i} &
        \pi_1(\Map(S^2,N),\beta)\rar[two heads]{ev_e} &
       \Stab\beta,
    \end{tikzcd}
    \]
    where $\Stab\beta\coloneqq\{g\in\pi_1N: g\beta=\beta\in\pi_2N\}$.
\end{lem}

Combining~\eqref{big-diagram} and Lemma~\ref{lem:htpy-grps} with Theorem~\ref{thm:main-Dax-part}, and denoting by $[U_*]$ the class of $U_*$ in $\pi_3N\cong\pi_1\Map_*(S^2,N)$, we have the commutative diagram:
\begin{equation}\label{big-diagram-final}\begin{tikzcd}[column sep=15pt]    
    \pi_3N\rar{j_*}\dar{i} &  \pi_1(\Map_*,\Emb_*, U_*)\rar\dar{\cong}[swap]{i^{rel}} &
    \pi_0\Emb_*(S^2,N)\rar[two heads]{p_*}\dar[two heads] &
    \pi_2N\dar[two heads]\\
    \pi_1(\Map(S^2,N),U)\rar{j}\dar[two heads]{ev_e} &  \pi_1(\Map,\Emb, U)\rar\dar{\cong}[swap]{\Da} &
    \pi_0\Emb(S^2,N)\rar[two heads]{p} &
    \pi_0\Map(S^2,N)\\
    \Stab[U_*] & \A &
    \end{tikzcd}
\end{equation}
which will imply desired results, as explained next.

\subsection{The proofs}\label{sec:dax-proofs}
The following recovers Theorem~\ref{thm:based-fq}. 
\begin{thm}\label{thm:main-Dax-part-1}
    There is a short exact sequence of sets
\[
\begin{tikzcd}
    \faktor{\A}{\Da\circ i^{rel}\circ j_*(\pi_3N)} \rar[tail] & \pi_0\Emb_*(S^2,N)\rar[two heads]{p_*} & 
    \pi_0\Map_*(S^2,N)\cong\pi_2N
\end{tikzcd}
\]
    and $\Da\circ i^{rel}\circ j_*=\phi_{[U_*]}$ from Equation~\eqref{eq-def:phi} of section~\ref{subsec:intro-based-isotopy}.
\end{thm}
\begin{proof}
    From diagram~\eqref{big-diagram-final} we have $\ker(p_*)=\coker(j_*)\cong\coker(\Da\circ i^{rel}\circ j_*)$, so it only remains to identify the last homomorphism. And indeed, for a class $A\in\pi_3N$ the element $j_*(A)\colon I\to\Map_*$ is a self-homotopy of $U$ that represents $A$ and $\Da(j_*(A))=\mu_3(j_*(A))$ by definition. Now, arguing as in the proof of Lemma~\ref{lem:mu-of-based-self-homotopies} we see that the track of $j_*(A)$ has $\mu_3(j_*(A))=\mu_3(A)+\lambda_N(A,U_*)$, so $\Da\circ i^{rel}\circ j_*(A)=\phi_{[U_*]}(A)$ as desired.
\end{proof}

Similarly, the following recovers Theorem~\ref{thm:free}.
\begin{thm}\label{thm:main-Dax-part-2}
    There is a short exact sequence of sets
\[
\begin{tikzcd}
    \left(\faktor{\A}{\phi_{[U_*]}(\pi_3N)}\right)_{s\mapsto {}^sa} \rar[tail] & \pi_0\Emb(S^2,N)\rar[two heads]{p} & 
    \pi_0\Map(S^2,N),
\end{tikzcd}
\]
where on the left we take the quotient by the action $s\mapsto {}^sa$ of $\Stab[U_*]$ from equation~\eqref{eq:aff} of section~\ref{intro:free-isotopy}.
\end{thm}

\begin{proof}
From diagram~\eqref{big-diagram-final} we have $\ker(p)=\coker(j)\cong\coker(\Da\circ j)$.  Using the leftmost column we can compute $\coker(\Da\circ j)$ in two steps
    \begin{enumerate}
        \item first take the cokernel of $\Da\circ j\circ i$,
        \item then mod out the induced action of $\Stab[U_*]$, using any section of $ev_e$.
    \end{enumerate}
    Note that the action in (2) is well defined, and the set of coinvariants is independent of the section, since in $\coker( j\circ i)$ we have modded out $\ker(ev_e)$.

    For (1), we simply note that $\Da\circ j\circ i=\Da\circ i^{rel}\circ j_*$ by the commutativity of the leftmost square in~\eqref{big-diagram-final}, and this is equal to $\phi_{[U_*]}$ by Theorem~\ref{thm:main-Dax-part-1}.

    For (2), to compute the action, we pick any section; by definition, this sends $s\in\Stab[U_*]$ to any
    \[
        J_s\in\pi_1(\Map(S^2,N),U),
    \]
    which we view as a free self-homotopy of $U_*$, for which $ev_e(J_s)=J_s(-,e)$ represents $s$.
    
    Then $s\in\Stab[U_*]$ acts by sending $a=\Da(H)$ to $\Da(J_s\cdot H)$. Since
    \[
        \Da(J_s\cdot H)=\Da(J_s)+s\Da(H)s^{-1}
    \]
    by Lemma~\ref{lem:basic-concatenation} (where $\mu_3$ notation was used in place of $\Da$), we see that the action of $s$ on $a$ is given by
    \[
        \Da(J_s)+sa s^{-1}=\mu_3(J_s)+sas^{-1}={}^sa
    \]
    as claimed.
\end{proof}

\vspace{0.5cm}

\phantomsection
\printbibliography

\vspace{1em}
\hrule
\vspace{0.2cm}

\end{document}